\documentclass{article}

\usepackage{mathptm}    % use Times fonts for text and math
\usepackage{mathdots}
\usepackage{latexsym}   % use all LaTeX fonts
\usepackage{amssymb}    % use all AMS fonts
\usepackage{amsmath}    % include some AMS-LaTeX functionality
\usepackage{graphics}
\usepackage{enumerate}
\usepackage{amsfonts}
\usepackage{tikz}
\usepackage{multicol}
\usepackage{amsthm}     % AMS style theorem and proof environment
\usetikzlibrary{automata,positioning,shapes,snakes}

\newtheorem{theorem}{Theorem}[section]
\newtheorem{lemma}[theorem]{Lemma}

\newtheorem{corollary}[theorem]{Corollary}

\newtheorem{remark}[theorem]{Remark}
\newtheorem{example}[theorem]{Example}
\newtheorem{definition}[theorem]{Definition}

\input diagxy

\def\GG{{\mathcal{G}}}
\def\ZZ{{\mathbb{Z}}}

\def\CP{{\mathcal P}}
\def\CL{{\mathcal L}}
\def\CD{{\mathcal D}}

\newcommand \Jac {\mathop{\rm Jac}}

\begin{document}

\begin{center}
\uppercase{\bf Critical Groups of Graphs with Dihedral Actions}

\vskip .25in

{\bf Darren B Glass}\\
{\small Department of Mathematics, Gettysburg College, Gettysburg PA 17325}\\
{\tt dglass@gettysburg.edu}\\

\vskip .1in

{\bf Criel Merino}\\
{\small Instituto de Matematicas, Universidad Nacional Autonoma de Mexico\\ C.U. Coyoacan 04510, Mexico,D.F. Mexico.}\\
{\tt merino@matem.unam.mex}\\

\end{center}

\begin{abstract}

In this paper we consider the critical group of finite connected graphs which admit harmonic actions by the dihedral group $D_n$.  In particular, we show that if the orbits of the $D_n$-action all have either $n$ or $2n$ points then the critical group of such a graph can be decomposed in terms of the critical groups of the quotients of the graph by certain subgroups of the automorphism group.  This is analogous to a theorem of Kani and Rosen which decomposes the Jacobians of algebraic curves with a $D_n$-action.
\end{abstract}

\section{Introduction}

In recent years, there have been a number of papers exploring the relationships between Jacobians of graphs (also known as critical groups) and Jacobians of algebraic curves.  This connection was originally written about by Baker and Norine in \cite{BN1} and has subsequently been explored in papers such as \cite{BN2}, \cite{BS}, \cite{Corry1}, \cite{Corry2}, and \cite{Primer}. While there is not a literal dictionary between the two situations, several authors have shown that variations of theorems in algebraic geometry such as the Riemann-Hurwitz formula, the Hurwitz bound on the size of automorphism groups, and the Riemann-Roch theorem work in the setting of graph theory.  At the same time, it is well known that the order of the critical group of a graph is the number of spanning trees on the graph, and recent papers such as \cite{CYZ}, \cite{ZY2},\cite{ZY} have given some results on the number of spanning trees of graphs which admit certain automorphisms.

Kani and Rosen proved the following result, which is actually a corollary to \cite[Theorem B]{KR}, describing the relationship between the Jacobian of a curve which admits certain automorphisms and the Jacobians of the quotients by these automorphisms.

\begin{theorem}\label{T:KR}
Let $X$ be an algebraic curve so that the dihedral group $D_n$ generated by two involutions $\sigma_1$ and $\sigma_2$ acts on $X$.  Then we have the following isogeny relationship:
\[\Jac(X) \oplus (\Jac(X/D_n))^2 \sim \Jac(X/\sigma_1) \oplus \Jac(X/\sigma_2) \oplus \Jac(X/\sigma_1\sigma_2)\]
\end{theorem}

In this paper, we explore whether an analogous theorem will hold for the Jacobians of graphs. In particular, given a graph $\GG$ with a dihedral action generated by two involutions $\sigma_1$ and $\sigma_2$, we investigate when it is the case that $\Jac(X/\sigma_1) \oplus \Jac(X/\sigma_2) \oplus \Jac(X/\sigma_1\sigma_2)$ is a subgroup of $\Jac(X) \oplus (\Jac(X/D_n))^2$. The following example shows that, in general, this will not be the case.

\begin{example}
Consider the following graph $\GG$:
\begin{center}
\begin{tikzpicture}
[shorten >=1pt,node distance=.8cm,on grid,auto,/tikz/initial text=]

\node[state](v1) at (0,0) {$v_1$};
\node[state] (x2) at (3,0) {$x_2$};
\node[state] (x1) at (3,3) {$x_1$};
\node [state](x3) at (3,-3) {$x_3$};
\node [state](v2) at (6,0) {$v_2$};
\path
(v1) edge [bend left=15] node {$a_1$} (x1)
(x1) edge [bend left=15] node {$d_1$} (v2)
(v1) edge [bend right=15] node {$a_2$} (x1)
(x1) edge [bend right=15] node {$d_2$} (v2)
(v1) edge [bend left=15] node {$b_1$} (x2)
(x2) edge [bend left=15] node {$e_1$} (v2)
(v1) edge [bend right=15] node {$b_2$} (x2)
(x2) edge [bend right=15] node {$e_2$} (v2)
(v1) edge [bend left=15] node {$c_1$} (x3)
(x3) edge [bend left=15] node {$f_1$} (v2)
(v1) edge [bend right=15] node {$c_2$} (x3)
(x3) edge [bend right=15] node {$f_2$} (v2);
\end{tikzpicture}
\end{center}

Define $\sigma_1$ to be the involution on this graph which permutes the edges by the map $(a_1 a_2)(b_1 c_2)(b_2 c_1)(d_1 d_2)(e_1 f_2)(e_2 f_1)$ and define $\sigma_2$ to be the involution which permutes the edges by the map $(a_1 b_2)(a_2 b_1)(c_1 c_2)(d_1 e_2)(d_2 e_1)(f_1 f_2)$.  One can check that $\sigma_1\sigma_2$ has order three and therefore that the group generated by $\sigma_1$ and $\sigma_2$ is isomorphic to $D_3$.  Moreover, the action of $D_3$ on $G$ is harmonic and the graph $G/D_3$ is a tree.

One can compute the quotient graphs and, by a straightforward calculation using the Smith Normal Form, compute that

\[K(G/\sigma_1) \cong K(G/\sigma_2) \cong \ZZ/12\ZZ\]
\[K(G/\sigma_1\sigma_2) \cong (\ZZ/2\ZZ)^2\]
\[K(G) \cong (\ZZ/2\ZZ)^2 \oplus (\ZZ/4\ZZ) \oplus (\ZZ/12\ZZ)\]

\noindent and in particular we conclude that $K(G/\sigma_1) \oplus K(G/\sigma_2) \oplus K(G/\sigma_1\sigma_2)$ is not a subgroup of $K(G)$.

\end{example}

We suspect that the problem with this example is that every element of the group $D_3$ fixes the two vertices at either end, and in particular the inertia groups at these vertices is all of $D_3$.  This is a situation that cannot occur for algebraic curves, where inertia groups must be cyclic except in the presence of wild ramification.  In this note, we will show that a theorem along the lines of Theorem \ref{T:KR} holds in the case where all of the inertia groups are either trivial or isomorphic to $\ZZ/2\ZZ$.

Theorems \ref{T:kernel} and \ref{T:quotient} will allow us to write down a decomposition of the group $K(G)$ in terms of the critical groups of the quotient graphs.  In general, it appears that these decompositions do not split, but under slightly stronger hypothesis, we can prove the following direct analogue to Theorem \ref{T:KR}.

\begin{theorem}\label{T:tree}
Let $n$ be odd and let $G$ be a graph such that the group $D_n$ generated by the two involutions $\sigma_1$ and $\sigma_2$ acts harmonically on $G$ and every $D_n$-orbit of $G$ has either $n$ or $2n$ points.  Furthermore, assume that the quotient graph $G/D_n$ is a tree.  Then $K(G/\sigma_1) \oplus K(G/\sigma_2) \oplus K(G/\sigma_1\sigma_2)$ is a subgroup of $K(G)$ and the quotient is isomorphic to $\ZZ/n\ZZ$.
\end{theorem}

More generally, one corollary of our results on the size of the critical groups will be the following:

\begin{corollary}
Let $G$ be a graph such that the group $D_n$ generated by the two involutions $\sigma_1$ and $\sigma_2$ acts harmonically on $G$ and every $D_n$-orbit of $G$ has either $n$ or $2n$ points.  Then the order of the critical group $K(G)$ is equal to $n\cdot |K(G/\sigma_1)| \cdot |K(G/\sigma_2)| \cdot |K(G/\sigma_1\sigma_2)|$.
\end{corollary}

The structure of the paper is as follows.  In Section \ref{S:notation} we define the setup of the paper and give some technical preliminaries.  In Section \ref{S:P123}, we consider the set of divisors that can be written as the sum of pullbacks of divisors on the three quotient graphs, and describe the structure of this set of divisors.  While this set is connected with the set $K(G/\sigma_1) \oplus K(G/\sigma_2) \oplus K(G/\sigma_1\sigma_2)$, they are not identical and the following two sections will consider in turn the extent to which this set is overcounting and then undercounting the set of divisors we are interested in.  A final section gives some examples and corollaries.

\section{Notation and Technical Preliminaries}\label{S:notation}

We recall that the critical group of a graph $G$ can be defined as the set of divisors on $G$ of total degree zero up to an equivalence relation.  This equivalence relation is often phrased in terms of a chip-firing game, in which any divisor can `borrow' (or `lend') a single chip from (or to) all of its neighbors simultaneously.

More precisely, given a finite connected graph $G$, we let $A$ be the adjacency matrix whose entry in row $i$ and column $j$ is the number of edges between vertex $i$ and vertex $j$ and $D$ be the diagonal matrix whose entry in row $i$ is the degree of vertex $i$.  We then define the Laplacian matrix of the graph to be $L = A-D$. It is easy to check that the entries of each column of $L$ sum to zero and in particular that $L$ has rank $n-1$.  Let $\CD$ be the set of divisors of degree zero on $G$, so that $\CD \cong \ZZ^{n-1}$ where $n$ is the number of vertices of $G$.  We define $\CL \subset \CD$  to be those divisors which are integral combinations of the columns of the Laplacian matrix $L$, and we define the critical group of $G$ to be the finite abelian group $\CD/\CL$.  The critical group is also called the Jacobian of $G$ and denoted by either $K(G)$ or $\Jac(G)$.  For details, we refer the reader to \cite{Biggs} and \cite{Primer}.

In this paper, we wish to consider graphs $\GG_n$ which admit a group action by the dihedral group $D_n$.  As is typical in the analogy between graph theory and algebraic geometry, we will require that the group action be harmonic in the sense of \cite{BN1}.  It is shown by Corry in \cite{Corry1} that the hypothesis that $D_n$ acts harmonically on $\GG_n$ is equivalent to the statement that no element of $D_n$ fixes both vertices in a given edge.  We will assume the additional hypothesis that every $D_n$-orbit of $\GG_n$ consists of either $n$ or $2n$ points.  In particular, the vertices of $\GG_n$ can be written as a disjoint union of sets $\{z_i^j\}_{i=1}^n$ for $j=1,\ldots s$ and $\{x_i^j,y_i^j\}_{i=1}^n$ for $j=1,\ldots,t$.  Without loss of generality, we may assume that $D_n$ is generated by two involutions $\sigma_1$ and $\sigma_2$ defined by

\begin{eqnarray*}
\sigma_1(z_i^j)=z_{n+1-i}^j & \sigma_1(x_i^j)=y_{n+1-i}^j & \sigma_1(y_i^j)=x_{n+1-i}^j\\
\sigma_2(z_i^j)=z_{n+2-i}^j & \sigma_2(x_i^j)=y_{n+2-i}^j & \sigma_2(y_i^j)=x_{n+2-i}^j
\end{eqnarray*}

\noindent where all addition in the subscripts takes place modulo $n$.  In particular, we note that it follows that for any vertex $v_i$ we have $\sigma_2(\sigma_1(v_i))=v_{i+1}$ and therefore that $\sigma_2\sigma_1$ is an element of order $n$.  We point out that in this notation the condition that $D_n$ acts harmonically is equivalent to the condition that our graph does not contain any edges between two vertices $z_i^{j_1}$ and $z_i^{j_2}$ in the same orbit of size $n$.

While the graphs that can be built in this way can be complicated in general, there are two families of such graphs that may be useful for the reader to keep in mind.

\begin{example}
Let $G$ be a graph with an automorphism $\phi$ so that $\phi^2=id$.  Fix two distinct vertices $a$ and $b$ so that $\phi(a)=b$.  In addition to those points, the graph $G$ will consist of $t$ pairs of vertices $(x^i,y^i)$ with $\phi(x^i)=y^i$ and $s$ vertices $z^j$ which are fixed under the involution $\phi$. Let $\GG_n$ be the graph obtained by taking $n$ copies of the graph $G$, which we will denote $G_i$, and identifying the vertex $b_i$ on $G_i$ with the vertex $a_{i+1}$ on $G_{i+1}$.
\end{example}

\begin{example} \label{Ex:circ}
Let $\{a_i\}$ be a set of $k$ integers so that $1 \le a_i < n$ for each $i$.  We define $G=C_n^{a_1,a_2,\ldots,a_k}$ to be the graph with $n$ vertices $v_1,\ldots,v_n$ which has edges connecting vertex $v_i$ to each vertex $v_{i+a_j}$, where addition works modulo $n$.  This family of graphs is called circulant graphs, and several examples are given in Figure \ref{F:circulant}.  Methods for counting the number of spanning trees (and hence the order of the critical group) of a circulant graph are discussed in \cite{AYI} and \cite{GL}.  We will see that our main theorem generalizes these results.

\begin{figure}
\begin{center}
\begin{multicols}{3}
\includegraphics[height=1in]{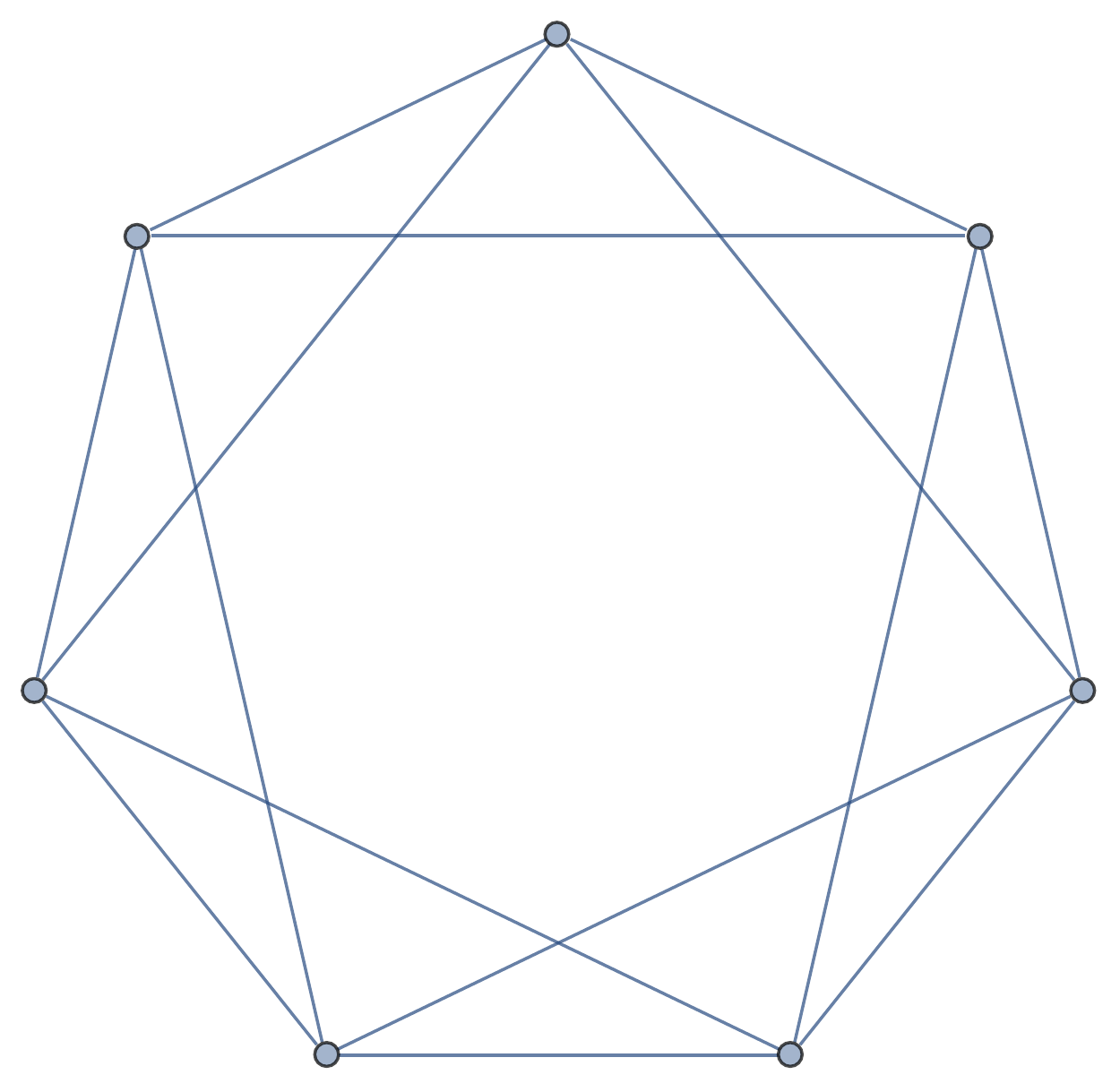}

\includegraphics[height=1in]{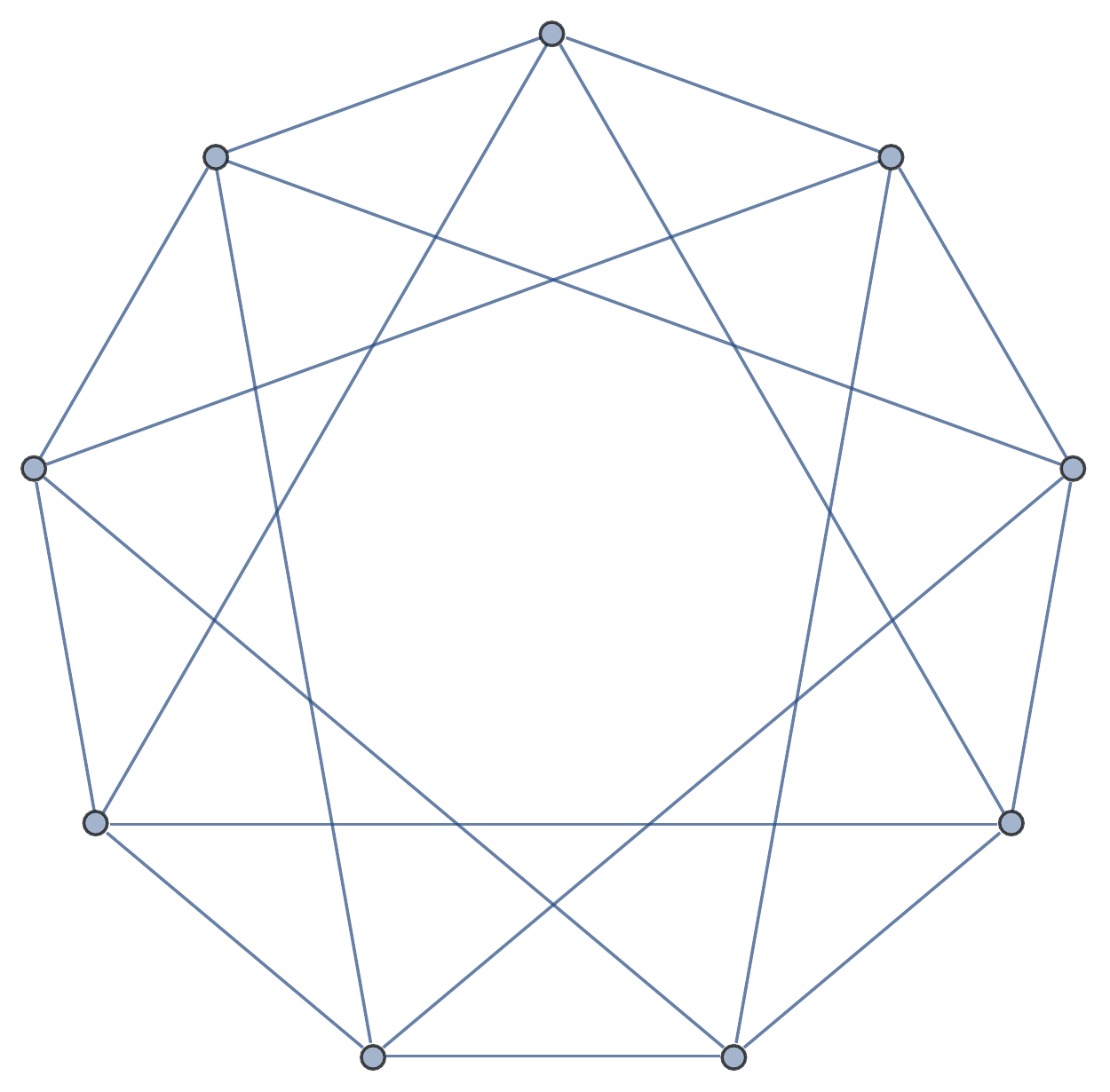}

\includegraphics[height=1in]{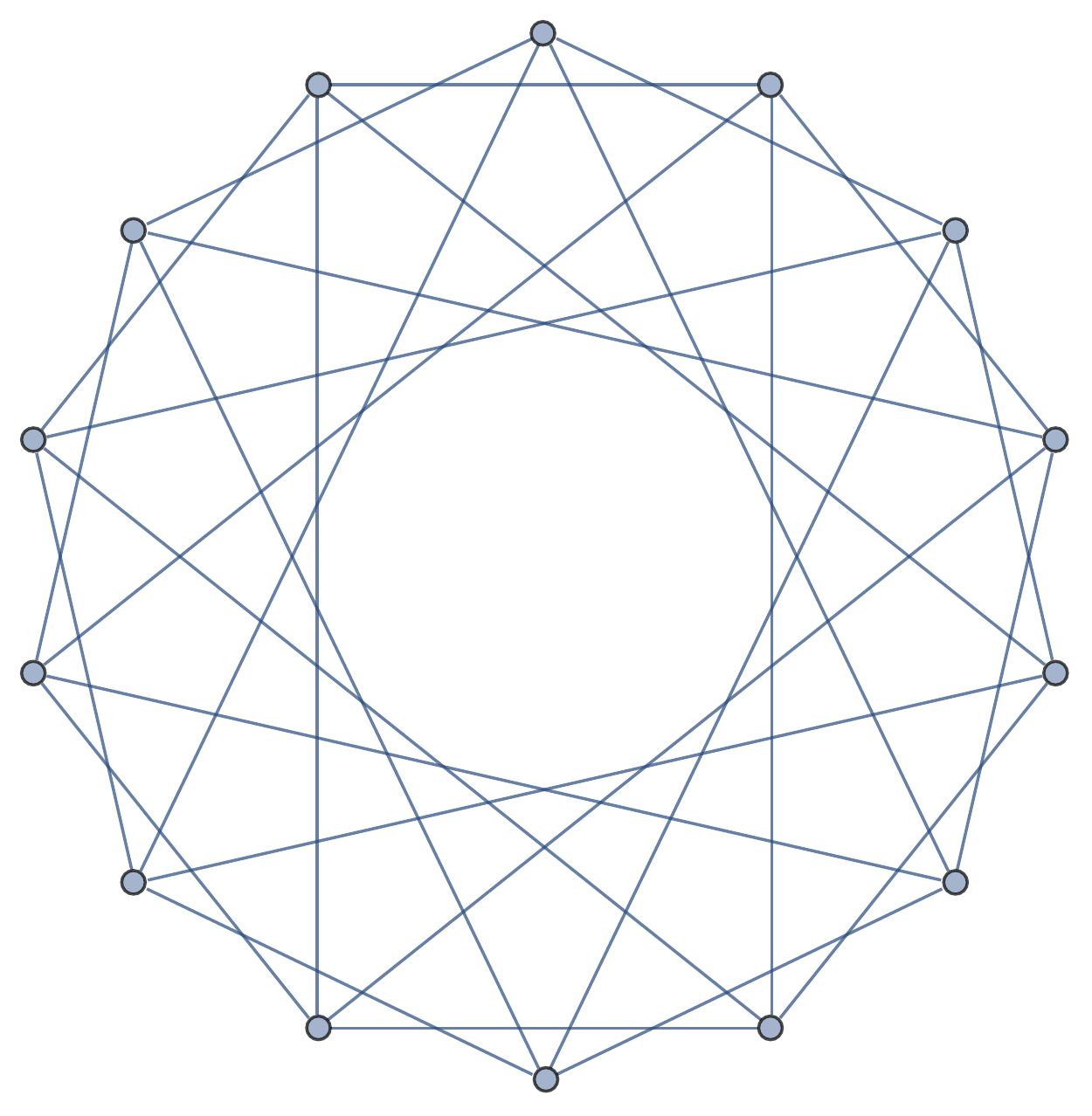}

\end{multicols}
\caption{The circulant graphs $C_7^{1,2}, C_9^{1,3}, C_{14}^{2,5}$}
\label{F:circulant}
\end{center}
\end{figure}

\end{example}

\section{Subgroups of Pullbacks} \label{S:P123}

Let $\GG_n$ be a graph with a dihedral action as described in the previous section.  Define $H_1$ to be the quotient graph $\GG_n/\sigma_1$.  Similarly, let $H_2 =\GG_n/\sigma_2$, and $H_3=\GG_n/\langle\sigma_1\sigma_2\rangle$.  In \cite{BN2}, Baker and Norine define the pullback map $\phi^*: \Jac(H_i) \rightarrow \Jac(\GG_n)$ as follows: given a divisor $\hat{D}$ on $H_i$ we define $\phi^*(\hat{D})$ to be the divisor on $\GG_n$ so that if $v$ is a vertex of $\GG_n$ with $\phi(v) = \hat{v}$ then $\phi^*(\hat{D})(v) = m_\phi(v)\hat{D}(\hat{v})$ where $m_\phi(v)$ is the horizontal multiplicity of the map at the vertex $v$.

In particular, Baker and Norine show that this map is not only well-defined but also injective, and thus that for each $i$, we can view $\phi^*(\Jac(H_i))$ as a subgroup of $\Jac(\GG_n)$. Throughout this note, we will typically omit the pullback map and write $\Jac(H_i)$ to denote its pullback as a subgroup of $\Jac(\GG_n)$. Given that each $\Jac(H_i)$ is a subgroup of $\Jac(\GG_n)$, we wish to determine whether the direct sum of the three is also a subgroup.  In order to do this, we will need to consider the intersections and sums of the three sets.  We begin by analyzing $\Jac(H_1)+\Jac(H_2)$.

Define $\CP_i$ to be the set of divisors on $\GG$ which can be obtained as the pullback of a divisor of degree zero on $H_i$.  The following lemma follows immediately from the definitions:

\begin{lemma}

The following conditions are necessary and sufficient for a divisor to be in $\CP_i$:
\begin{itemize}
\item A divisor $\delta$ is in $\CP_1$ if and only if $\delta(x^j_i)=\delta(y^j_{n+1-i})$ for all $i,j$ and, if $n$ is odd then $\delta(z^j_\frac{n+1}{2})$ is even for all $j$.
\item A divisor $\delta$ is in $\CP_2$ if and only if $\delta(x^j_i)=\delta(y^j_{n+2-i})$ for all $i,j$, $\delta(z^j_1)$ is even for all $j$ and, if $n$ is even then $\delta(z^j_\frac{n+2}{2})$ is even for all $j$.
\item A divisor $\delta$ is in $\CP_3$ if and only if for each fixed $j$ the values of $\delta(x_i^j)$, $\delta(y_i^j)$, and $\delta(z_i^j)$ are each constant as $i$ varies.
\end{itemize}
\end{lemma}

We wish to consider the set $\CP_1+\CP_2$.

\begin{theorem}\label{T:P12}
A divisor $\delta$ of degree $0$ on $\GG_n$ can be written as the sum of a divisor $\delta_1 \in \CP_1$ and a divisor $\delta_2 \in \CP_2$ if and only if the following conditions all hold:

\begin{enumerate}[(1)]
\item $\sum_{i=1}^n i(X_i+Y_i+Z_i) \equiv 0$ mod $n$, where $X_i$ (resp. $Y_i, Z_i$) is the sum of the values of the divisor on all vertices $x^j_i$ (resp. $y^j_i, z^j_i$).
\item For each orbit of order $2n$ given by $\{x_i,y_i\}$, we have that $\sum_i \delta(x_i) = \sum_i \delta(y_i)$.
\item For each orbit of order $n$ given by $\{z_i\}$ we have that $\sum_i \delta(z_i) \equiv 0$ mod $2$.
\end{enumerate}
\end{theorem}

\begin{proof}
Let $\delta \in \CP_1+\CP_2$. Consider an orbit in $\GG_n$ consisting of $2n$ points $x_1,y_1,\ldots,x_n,y_n$.  Then a divisor $\delta$ is in $\CP_1+\CP_2$ only if it can be written as $\delta=\delta_1+\delta_2$ where $\delta_1(x_i)=\delta_1(y_{n+1-i})$ and $\delta_2(x_i)=\delta_2(y_{n+2-i})$.  In particular, this implies that $\sum_i \delta(x_i) = \sum_i \delta(y_i)$.  Similarly, if one considers an orbit consisting of $n$ points $z_1,\ldots,z_n$ then we have $\delta_1(z_i)=\delta_1(z_{n+1-i})$ and $\delta_2(z_i)=\delta_2(z_{n+2-i})$.  Additionally, we have the restriction that $\delta_2(z_1)$ is even, as is $\delta_1(z_{\frac{n+1}{2}})$ (resp. $\delta_2(z_{\frac{n+2}{2}})$) if $n$ is odd (resp. if $n$ is even). In particular, one can deduce that $\sum_i\delta(z_i)$ is even.

To show that the third condition is necessary, we begin by considering a divisor $\delta_1 \in \CP_1$.  The fact that $\delta_1(x^j_i)=\delta_1(y^j_{n+1-i})$ implies that the sum $X_i$ will be equal to the sum $Y_{n+1-i}$ for all $i$.  We also have that $Z_i=Z_{n+1-i}$ and that $Z_{\frac{n+1}{2}}$ is even (if $n$ is even we define $Z_{\frac{n+1}{2}}=0$).  Finally, the fact that $\delta_1$ is a pullback of a degree zero divisor tells us that $\displaystyle \sum_{i=1}^n X_i +\sum_{i=1}^{\lfloor n/2 \rfloor}Z_i + \frac{1}{2}Z_{\frac{n+1}{2}}=0$.  We can then compute that

\begin{eqnarray*}
& & \sum_{i=1}^n i(X_i+Y_i+Z_i) \\
&=&\sum_{i=1}^n (i+(n+1-i)X_i + \sum_{i=1}^{\lfloor n/2 \rfloor}(i+(n+1-i))Z_i + \frac{n+1}{2}Z_{\frac{n+1}{2}}\\
&=&\sum_{i=1}^n (n+1)X_i +\sum_{i=1}^{\lfloor n/2\rfloor}(n+1)Z_i + \frac{n+1}{2}Z_{\frac{n+1}{2}}\\
&=&(n+1)(\sum_{i=1}^n X_i +\sum_{i=1}^{\lfloor n/2 \rfloor}Z_i + \frac{1}{2}Z_{\frac{n+1}{2}})\\
&=&0
\end{eqnarray*}

To see that a divisor $\delta_2 \in \CP_2$ will satisfy the global condition, we note that for such a divisor we have $X_i=Y_{n+2-i}$, $Z_i=Z_{n+2-i}$, $Z_1$ even and $Z_{\frac{n+2}{2}}$ even (where we define $Z_{\frac{n+2}{2}} = 0$ if $n$ is odd). Additionally, the restriction that $\delta_2$ is a pullback of a degree zero divisor tells us that $\displaystyle \sum_{i=1}^n X_i +\sum_{i=2}^{\lfloor \frac{n}{2} \rfloor}Z_i + \frac{1}{2}(Z_1+Z_{\frac{n+2}{2}})=0$.  We can now compute:

\begin{eqnarray*}
& &\sum_{i=1}^n i(X_i+Y_i+Z_i) \\
&=&\sum_{i=1}^n (i+n+2-i)X_i + Z_1+ \sum_{i=2}^{\lfloor n/2\rfloor}(i+n+2-i)Z_i + \frac{n+2}{2}Z_{\frac{n+2}{2}}\\
&=&\sum_{i=1}^n (n+2)X_i +Z_1 + \sum_{i=2}^{\lfloor n/2 \rfloor}(n+2)Z_i + \frac{n+2}{2}Z_{\frac{n+2}{2}}\\
&\equiv &(n+2)(\sum_{i=1}^n X_i +\sum_{i=2}^{\lfloor n/2\rfloor}Z_i + \frac{1}{2}(Z_1+Z_{\frac{n+2}{2}})) \text{ (mod $n$)}\\
& \equiv &0 \text{ (mod $n$)}
\end{eqnarray*}

Because these conditions are linear and any divisor in $\CP$ can be written as the sum of divisors in the $\CP_i$, we have therefore shown that all divisors in $\CP_1+\CP_2$ satisfies this condition as well.

We next wish to show that the three conditions are sufficient. In particular, assume that $\delta$ is a divisor of degree zero satisfying the three conditions of the theorem, and let $a= \frac{1}{n}\sum_i i(X_i+Y_i+Z_i)$.  We define divisors $\delta_1$ and $\delta_2$ as follows:

\begin{eqnarray*}
\delta_1(x_i^j) &=& (\delta(x_1^j) + \ldots + \delta(x_i^j)) - (\delta(y_n^j)+\ldots + \delta(y_{n+2-i}^j)) + a_j\\
\delta_1(y_i^j) &=& \delta_1(x_{n+1-i}^j) \\
\delta_1(z_i^j) &=& (\delta(z_1^j)+ \ldots +\delta(z_i^j)) - (\delta(z_n^j)+ \ldots +\delta(z_{n+2-i}^j)) + b_j \\
\delta_2(x_i^j) &=& (\delta(y_n^j)+\ldots + \delta(y_{n+2-i}^j) - (\delta(x_1^j) + \ldots + \delta(x_{i-1}^j)) - a_j\\
\delta_2(y_i^j) &=& \delta_2(x_{n+2-i}^j) \\
\delta_2(z_i^j) &=& (\delta(z_n^j)+ \ldots +\delta(z_{n+2-i}^j))-(\delta(z_1^j)+ \ldots +\delta(z_{i-1}^j)) - b_j \\
\end{eqnarray*}

\noindent where $a_1=a$ and all other $a_j=b_j=0$ if there is at least one orbit of size $2n$ and $b_1=2a$ and all other $b_j=0$ if all orbits have size $n$.

It is clear by definition that $\delta_1$ and $\delta_2$ have the necessary symmetry conditions to be in $\CP_1$ and $\CP_2$, and straightforward computations shows that $\delta_1 + \delta_2 = \delta$ and that the necessary $\delta_1(z_i^j)$ and $\delta_2(z_i^j)$ are even. It remains only to show that $\delta_1$ and $\delta_2$ each have degree zero.  To see the former, we note that for any fixed $j$ we have:

\begin{eqnarray*}
\sum_i \delta_1(z_i^j) &=&  nb_j + \sum_i (n+2-2i) \delta(z_i^j)\\
\sum_i \delta_1(x_i^j) + \sum_i \delta_1(y_i^j) &=&  2na_j + \sum_i (n+2-2i)(\delta(x_i^j) + \delta(y_i^j)) \\
\end{eqnarray*}

In particular, we have that

\begin{eqnarray*}
\sum_{v \in \GG_n} \delta_1(v) &=& 2na + \sum_i (n+2-2i)(X_i+Y_i+Z_i)\\
&=& 2na + (n+2)(deg(\delta)) + \sum_i (-2i)(X_i+Y_i+Z_i) \\
&=& 2na-2na\\
&=&0
\end{eqnarray*}

It follows that $\delta_1 \in \CP_1$.  Furthermore, because $\delta$ and $\delta_1$ have total degree zero and $\delta_2 = \delta-\delta_1$ it is immediate that $\delta_2 \in \CP_2$.

\end{proof}

Let $\CP = \CP_1 + \CP_2 + \CP_3$ be the set of divisors on $\GG$ which can be written as a sum of divisors that are pullbacks of divisors of degree zero on $H_1, H_2,$ and $H_3$.

\begin{theorem}\label{T:P123odd}
Let $n$ be odd and let $\delta$ be a divisor of degree zero on $\GG_n$.  Then $\delta \in \CP$ if and only if the following conditions hold.
\begin{enumerate}[(1)]
\item $\sum_{i=1}^n i(X_i+Y_i+Z_i) \equiv 0$ mod $n$, where $X_i$ (resp. $Y_i, Z_i$) is the sum of the values of the divisor on all vertices $x^j_i$ (resp. $y^j_i, z^j_i$). \label{Global}
\item For all orbits $\{x_i,y_i\}$ of size $2n$ we have $\sum_i \delta(x_i) \equiv \sum_i \delta(y_i)$ (mod $n$).
\end{enumerate}
\end{theorem}

\begin{proof}
We note that it follows from Theorem \ref{T:P12} that divisors in $\CP_1+\CP_2$ satisfy the above conditions. We next consider a divisor $\delta_3 \in \CP_3$.  For any fixed $j$, $\delta_3(x^j_i)$ will be constant for all $i$, as will $\delta_3(y^j_i)$ and $\delta_3(z^j_i)$.  Thus it is easy to see that $\sum_{i=1}^n x^j_i \equiv \sum_{i=1}^n y^j_i \equiv 0$ (mod $n$). Moreover, the sums $X_i, Y_i$ and $Z_i$ will be independent of $i$ and one can quickly deduce that $\delta_3$ also satisfies the global condition.

We further claim that these conditions are also sufficient to define $\CP$.  In particular, let $\delta$ be a divisor satisfying these conditions.  Let $a_j = \frac{1}{n}\sum_{i=1}^n (\delta(x^j_i)-\delta(y^j_i))$.  Define a divisor $\delta_3$ so that $\delta_3(x_i^j)=a_j$ for all $i$ and $\delta_3(y_i^j)=0$ for all $i,j$. For all $j \ge 2$ we define $\delta_3(z_i^j)=0$ if $\sum_i \delta(z_i^j)$ is even and $\delta_3(z_i^j)=1$  if $\sum_i \delta(z_i^j)$ is odd.  Finally, for all $i$ we define $\delta_3(z_i^1)=-(\sum_j a_j)- \gamma$ where $\gamma$ is the number of $j \ge 2$ for which $\sum_i \delta(z_i^j)$ is odd.  None of these depend on the choice of $i$ and moreover it is clear that $\sum_{v \in \GG_n} \delta_3(v) = 0$ so it follows that $\delta_3 \in \CP_3$.

Let $\delta_{12}=\delta-\delta_3$. One can easily check that $\sum_{v \in \GG_n} \delta_{12}(v)=0$ and that for any $j$ we have $\sum_i \delta_{12}(x_i^j)=\sum_i \delta_{12}(y_i^j)$. It is also clear that for $j \ge 2$ we have that $\sum_i \delta_{12}(z_i^j)$ is even.  For $j=1$, we compute:

\begin{eqnarray*}
\sum_{i=1}^n \delta_{12}(z_i^1) &=& \sum_{i=1}^n \delta(z_i^1) +n \sum_j a_j + n \gamma \\
&=& \sum_{i=1}^n \delta(z_i^1) + \sum_{i,j}(\delta(x_i^j)-\delta(y_i^j)) + n \gamma \\
& \equiv & \sum_{i=1}^n \delta(z_i^1) + \sum_{i,j}(\delta(x_i^j)+\delta(y_i^j)) +  \gamma\\
&=& \sum_{v \in \GG_n} \delta(v) \\
&=& 0\\
\end{eqnarray*}

\noindent where the congruence is mod $2$, showing that $\sum_{i=1}^n \delta_{12}(z_i^1)$ is even as well. Finally, one can see that, given that both $\delta$ and $\delta_3$ satisfy the global condition $(\ref{Global})$ it must also be the case that $\delta_{12}$ does as well. It follows from Theorem \ref{T:P12} that $\delta_{12}$ can be expressed as the sum of a divisor $\delta_1 \in \CP_1$ and a divisor $\delta_2 \in \CP_2$.  This proves the theorem.
\end{proof}

\begin{theorem}\label{T:P123even}
Let $n$ be even and let $\delta$ be a divisor of degree zero on $\GG_n$.  Then $\delta \in \CP$ if and only if the following conditions hold.
\begin{enumerate}[(1)]
\item $\sum_{i=1}^n i(X_i+Y_i+Z_i) \equiv 0$ mod $n$, where $X_i$ (resp. $Y_i, Z_i$) is the sum of the values of the divisor on all vertices $x^j_i$ (resp. $y^j_i, z^j_i$).
\item For each orbit of order $2n$ given by $\{x_i,y_i\}$ we have $\sum_i \delta(x_i) \equiv \sum_i \delta(y_i)$ (mod $n$).
\item For each orbit of order $n$ given by $\{z_i\}$ we have $\sum_i \delta(z_i) \equiv 0$ mod $2$. \label{zeven}
\end{enumerate}
\end{theorem}

\begin{proof}
We begin by noting that the argument that the first two conditions are necessary in the proof of Theorem \ref{T:P123odd} carries over exactly to this case.  Moreover, for any divisor $\delta_3 \in \CP_3$ and any orbit of order $n$ we have $\sum_i \delta(z_i) = n\delta(z_1) \equiv 0$ mod $2$.  Because this condition is also true for divisors in $\CP_1+\CP_2$ by Theorem \ref{T:P12} then it must be necessary for divisors in $\CP$.

To see that these conditions are sufficient, we let $\delta$ be a divisor satisfying them and, as above, we set $a_j = \frac{1}{n}\sum_{i=1}^n (\delta(x^j_i)-\delta(y^j_i))$.  Define a divisor $\delta_3$ as follows:
\[\delta_3(x_i^j)=a_j \text{ for all }i\]
\[\delta_3(y_i^j)=0\text{ for all }i,j\]
\[\delta_3(z_i^1)=-(\sum_j a_j)\]
\[\delta_3(z_i^j)=0 \text{ if }j \ge 2\]

One can easily check that $\delta_3 \in \CP_3$.  Moreover, if we set $\delta_{12}=\delta-\delta_3$ then it is straightforward to see that $\delta_{12}$ satisfies the conditions of Theorem \ref{T:P12} and is therefore in $\CP_1+\CP_2$.  The theorem follows.
\end{proof}

The following corollary is an immediate consequence.

\begin{corollary} \label{C:DP}
Let $\CD$ be the set of divisors of degree zero on $\GG_n$.  If $n$ is odd then $\CD/\CP \cong (\ZZ/n\ZZ)^{t+1}$.  If $n$ is even then $\CD/\CP \cong (\ZZ/n\ZZ)^{t+1} \oplus (\ZZ/2\ZZ)^{s-1}$
\end{corollary}

\section{Intersections of the sets $\CP_i$}

In this section, we consider the relationships between the divisors in $\Jac(H_1) \oplus \Jac(H_2) \oplus \Jac(H_3)$ and the divisors in $\CP$.  In particular, one can define a natural surjective map \[\Jac(H_1) \oplus \Jac(H_2) \oplus \Jac(H_3) \rightarrow \Jac(H_1) + \Jac(H_2) + \Jac(H_3)\]
This map will be injective if and only if the sets $\CP_1, \CP_2, \CP_3$ are all disjoint.  In this section, we consider these intersections and their implications about the kernel of this map.  In order to do this, we first prove the following lemma, which actually applies to a more general situations than the graphs $\GG_n$ defined in Section \ref{S:notation}.

\begin{lemma}\label{L:orbits}
Let $G$ be a graph and let $\Gamma$ be a group which acts harmonically on $G$ so that the quotient $G/\Gamma$ is a graph $\hat{G}$.  Let $\delta$ be a divisor of degree zero on $G$.  Then $\delta$ is the pullback of a divisor $\hat{\delta}$ of degree zero on $\hat{G}$ if and only if the following two conditions hold:
\begin{itemize}
\item $\delta$ is constant on $\Gamma$-orbits of $G$.  In particular, $\delta(v)=\delta(g\cdot v)$ for all $g \in \Gamma$.
\item For all vertices $v$, the value $\delta(v)$ is a multiple of ${\displaystyle \frac{|\Gamma|}{O_v}}$, where $O_v$ is the number of points in the orbit $\Gamma v$.
\end{itemize}
\end{lemma}

\begin{proof}
The necessity of these conditions follows immediately from the definition of the pullback of a divisor. To prove their sufficiency, assume that $\delta$ is a divisor satisfying them.  Define $\hat{\delta}$ to be a divisor on $\hat{G}$ so that $\hat{\delta}(\hat{v}) = \frac{O_v}{|\Gamma|} \delta(v)$ where $v$ is any vertex of $G$ lying above $\hat{v}$.

Recall that the horizonal multiplicity of the map $\phi:G \rightarrow \hat{G}$ at a vertex $v \in V(G)$ is defined to be the size of the inertia group of the action at the vertex $v$:
\[m_\phi(v) = |I(v)| = |\{g \in \Gamma | g \cdot v = v\}|\]

In particular, it follows from the orbit-stabilizer theorem that $m_\phi(v)=|\Gamma|/O_v$ for each $v$. By definition of the pullback map, for each $v \in V(G)$ lying above $\hat{v}$ we have: \begin{eqnarray*}
\phi^*(\hat{D})(v) &=& m_\phi(v)\hat{D}(\hat{v})\\
&=&\frac{|\Gamma|}{O_v}\hat{D}(\hat{v})\\
&=&\frac{|\Gamma|}{O_v} \frac{O_v}{|\Gamma|} D(v) \\
&=& D(v)\\
\end{eqnarray*}
\noindent proving the lemma.
\end{proof}

\begin{remark}
In particular, this lemma shows that if $\hat{G}$ is a tree then any divisor $\delta$ which satisfies the hypotheses of the lemma is equivalent to the zero divisor in $\Jac(G)$. We can make this constructive in the following way:  Choose some leaf $v_0$ of the tree $G/\Gamma$. There will be $|\Gamma|$ edges in $G$ which lie above the edge coming out of this leaf and each of the vertices above this leaf will have $|\Gamma|/m$ edges coming from it.  Letting $\delta(x) = k(|\Gamma|/m)$ for each vertex $x$ lying above $v$ we see that if we fire all of the vertices in this orbit $k$ times then we will reduce the number of chips on each of these vertices to zero.  Note that if $\delta(x)$ were not a multiple of $(|\Gamma|/m)$ this process could not reduce the number of chips to zero.

We proceed by induction, next choosing a leaf $v_1$ of the tree obtained by deleting $v_0$ from the tree $G/\Gamma$ and firing all of the vertices that lie above it the same number of times to reduce the number of chips that lie above it to zero.  We note that if $v_1$ is adjacent to $v_0$ then every time we fire the vertices above $v_1$ we will also fire the vertices above $v_0$ in order to maintain the number of chips above that vertex to $0$.  In this way, we can eventually reduce the number of chips on all vertices to zero.
\end{remark}

Returning to our situation, we recall that the sets $\CP_1$ and $\CP_2$ are the sets of divisors on $\GG_n$ which are pullbacks of divisors of degree zero on the graphs $H_1$ and $H_2$ which are obtained by taking the quotient of $\GG_n$ by an involution.  In particular, one can check that a divisor $D$ is in the intersection $\CP_1 \cap \CP_2$ if and only if it is constant on a given $D_n$ orbit and has an even value at vertices in orbits with $n$ points.  These conditions are identical to those in Lemma \ref{L:orbits}, and therefore $D$ is a pullback of a divisor on $\hat{G}$.  In particular, we have shown the following:

\begin{theorem}\label{T:P1P2}
Let $\psi:\Jac(\hat{G}) \rightarrow \Jac(H_1) \oplus \Jac(H_2)$ be defined by $\psi(g) = (g,-g)$ and let $\phi:\Jac(H_1) \oplus \Jac(H_2) \rightarrow \Jac(H_1) + \Jac(H_2)$ be defined by $\phi(h_1,h_2)=h_1+h_2$.  Then the following sequence is exact.

\[1 \rightarrow \Jac(\hat{G}) \overset{\psi}{\rightarrow} \Jac(H_1) \oplus \Jac(H_2) \overset{\phi}{\rightarrow} \Jac(H_1) + \Jac(H_2) \rightarrow 1 \]

\end{theorem}

It remains to consider the intersection of the set $\CP_1 + \CP_2$ with the set $\CP_3$.  By doing so, we are able to prove the following theorem.

\begin{theorem}\label{T:kernel}
Let $\psi$ be the natural surjection from $\Jac(H_1) \oplus \Jac(H_2) \oplus \Jac(H_3)$ to $\Jac(H_1) + \Jac(H_2) + \Jac(H_3)$.
\begin{itemize}
\item If $n$ is odd or if $\GG_n$ has no $D_n$-orbits of size $n$ then the kernel of $\psi$ is isomorphic to $\Jac(\hat{G})^2$.
\item If $n$ is even and $s \ge 1$ is the number of $D_n$-orbits of $\GG_n$ then the kernel of $\psi$ is isomorphic to $\Jac(\hat{G})^2 \oplus (\ZZ/2\ZZ)^{s-1}$.
\end{itemize}
\end{theorem}

\begin{proof}
Recall that a divisor $\gamma \in \CP_1 + \CP_2$ if and only if it satisfies the requirements of Theorem \ref{T:P12}.  Moreover, $\gamma \in \CP_3$ if and only if for each fixed $j$, the value of $\gamma(x^j_i)$ is constant as $i$ changes and the same is true for the sets of $y^j_i$ and $z^j_i$. One can check that these conditions together imply that $\gamma \in (\CP_1+\CP_2) \cap \CP_3$ if and only if the value of the divisor $\gamma$ is constant on all $D_n$-orbits and that for each orbit of size $n$ we have that $n \cdot \gamma(z^j)$ is even.

If $n$ is odd, this implies that the value of $\gamma(z^j)$ is even for each $j$ and therefore $\gamma$ satisfies the conditions of Lemma \ref{L:orbits}, so it is the pullback of a divisor on $\hat{G}$ in $K(\GG_n)$.  Combining this result with Theorem \ref{T:P1P2} proves the first half of the theorem.

On the other hand, if $n$ is even then the values of the divisor at the vertices $z^j_i$ can be odd and therefore may not be the pullbacks of a divisor in $\hat{G}$.  We note first that if $s=0$ then this restriction is vacuous, and if $s=1$ then any divisor in $(\CP_1 + \CP_2) \cap \CP_3$ must have an even value at the vertices $z_i^1$ as the total degree of the divisor is equal to zero.

If $s \ge 2$ then for each $1 \le j \le s-1$ define $\gamma_j$ to be the divisor on $\GG_n$ whose values are given by $\delta_j(z_i^j) = 1$ and $\delta_j(z_i^s)=-1$ for each $i$ and $\delta_j(v)=0$ for all other vertices $v$.  We make the following claims about the divisors $\delta_j$, all of which follow easily from Lemma \ref{L:orbits}:

\begin{itemize}
\item For each $j$, $\delta_j \in (\CP_1 + \CP_2) \cap \CP_3$ but $\delta_j$ is not the pullback of a divisor on $\hat{G}$.
\item For each $j$, $2\delta_j$ is a pullback of a divisor on $\hat{G}$.
\item If ${\displaystyle \sum_{j=1}^{s-1} a_j \delta_j}$ is the pullback of a divisor on $\hat{G}$ then all of the $a_j$ must be even.
\item Any divisor $\gamma \in (\CP_1 + \CP_2) \cap \CP_3$ can be written uniquely as ${\displaystyle \gamma = \hat{\gamma} + \sum_{j=1}^{s-1} a_j\delta_j}$ where $a_j \in \{0,1\}$ and $\hat{\gamma}$ is the pullback of a divisor on $\hat{G}$.
\end{itemize}

The second half of the theorem follows by combining these observations

\end{proof}

We note that $\Jac(H_1) + \Jac(H_2) + \Jac(H_3)$ is the subgroup of $\Jac(\GG_n)$ consisting of all divisors that can be written as the sum of divisors which are pullbacks of divisors on the three quotients.  In particular, we see that if $n$ is odd and $\hat{G}$ is a tree, then $\Jac(H_1) \oplus \Jac(H_2) \oplus \Jac(H_3) \subseteq \Jac(\GG_n)$, proving the first part of Theorem \ref{T:tree}.  In the next section, we will determine the quotient of this map.

\section{Cokernel of Map}

In this section, we wish to describe the divisors of degree zero on $\GG_n$ which cannot be obtained by the methods in the previous section.  In particular, we want to compute the cokernel of the map
\[\Jac(H_1)+\Jac(H_2)+\Jac(H_3) \rightarrow \Jac(\GG_n) \]

We begin by noting that, as a set, $\Jac(H_1)+\Jac(H_2)+\Jac(H_3)$ can be thought of as the set of divisors $\CP$.  We wish to consider this set up to an equivalence relation induced by the three separate pullbacks.  To be explicit, for $i=1,2$, $\Jac(H_i)$ consists of the set of divisors which are pullbacks of divisors of degree zero on $H_i$ up to the equivalence that two such divisors are the same if one can be obtained from the other by firing vertices that are fixed under the automorphism $\sigma_i$ or by simultaneously firing pairs of vertices that are interchanged by $\sigma_i$.  On the other hand, $\Jac(H_3)$ consists of divisors in $\CP_3$ up to the equivalence that we are allowed to simultaneously fire all vertices $x^j_i$ (or $y_i^j$ or $z_i^j$) for a fixed $j$. In particular, we wish to define the following set:

\begin{definition}
Let $\CL'$ be the subspace of $\CL$ generated by the divisors representing:
\begin{itemize}
\item Firing any single vertex $a_i$.
\item Firing any single vertex $z_i^j$.
\item Firing any pair of vertices $x_{i_1}^j$ and $y_{i_2}^j$ simultaneously.
\item Firing all vertices $x_i^j$ for a fixed $j$ simultaneously.
\item Firing all vertices $y_i^j$ for a fixed $j$ simultaneously.
\end{itemize}
\end{definition}

The following theorem is immediate from the definition:

\begin{theorem}\label{T:Hstructure}
$\Jac(H_1)+\Jac(H_2)+\Jac(H_3) = \CP/\CL'$
\end{theorem}

We note that $\CL' \subseteq \CL \cap \CP$ and in many cases these two sets are actually equal.  However, there are examples in which $\CL'$ is a proper subset, which occurs when the divisor corresponding to firing a single $x_i^j$ or $y_i^j$ vertex from a given orbit is in $\CP$.

\begin{lemma} \label{L:LL'}
$\CL/\CL' \cong (\ZZ/n\ZZ)^t$
\end{lemma}

\begin{proof}
$\CL$ is generated by basis elements corresponding to firing the vertices $x_i^j,y_i^j,z_i^j$ and $a_i$.  The basis elements corresponding to the $z_i^j$ and $a_i$ are also contained in $\CL'$. Firing the vertex $x_i^j$ a single time is not in $\CL'$, but firing $n$ times is, making it an element of order $n$ in $\CL/\CL'$.  Moreover, $x_i^j$ and $y_i^j$ are inverses in $\CL/\CL'$.
\end{proof}

We have already noted that $\Jac(H_1) + \Jac(H_2) + \Jac(H_3)$ is a subgroup of $\Jac(\GG_n)$, and we are interested in computing the quotient.  The following theorem allows us to do this.

\begin{theorem}\label{T:quotient}
$\Jac(\GG_n)/(\Jac(H_1) + \Jac(H_2) + \Jac(H_3)) \cong \ZZ/n\ZZ$ if $n$ is odd or $s=0$ and $(\ZZ/n\ZZ) \oplus (\ZZ/2\ZZ)^{s-1}$ if $n$ is even and $s \ge 1$.
\end{theorem}

\begin{proof}
We have the following inclusion of groups:

\[ \bfig \morphism(100,0)/{^(->}/<0,300>[\CL'`\CP \cap \CL;]
\morphism(100,300)/{^(->}/<-300,300>[\CP \cap \CL`\CP;]
\morphism(100,300)/{^(->}/<300,300>[\CP \cap \CL`\CL;]
\morphism(-200,600)/{^(->}/<300,300>[\CP`\CP+\CL;]
\morphism(400,600)/{^(->}/<-300,300>[\CL`\CP+\CL;]
\morphism(100,900)/{^(->}/<0,300>[\CP+\CL`\CD;]
\efig \]

By definition, $\Jac(\GG_n) \cong \CD/\CL$ and by Theorem \ref{T:Hstructure} we have $\Jac(H_1)+\Jac(H_2)+\Jac(H_3) = \CP/\CL'$.  Therefore, we are interested in computing $(\CD/\CL)/(\CP/\CL')$.  In order to do this, we note that the following relationships follow immediately from the isomorphism theorems:
 \[\CD/(\CP+\CL) \cong (\CD/\CL)/((\CP+\CL)/\CL)\]
 \[(\CP+\CL)/\CL \cong \CP/(\CP \cap \CL)\]
 \[\CP/(\CP \cap \CL) \cong (\CP/\CL')/(\CP \cap \CL/\CL')\]
\noindent from which one concludes that $(\CD/\CL)/(\CP/\CL') \cong (\CD/(\CP+\CL))/(\CP \cap \CL/\CL')$.  By symmetry, we see that $(\CD/\CL)/(\CP/\CL') \cong (\CD/\CP)/(\CL/\CL')$.  The theorem then follows from Corollary \ref{C:DP} and Lemma \ref{L:LL'}.
\end{proof}

\section{Examples}

In this section, we give several examples of our main results and how they can be used to compute the Jacobians of graphs admitting certain symmetry groups.

\subsection{A Circulant Graph}

For our first example, let $n=2k+1 \ge 3$ be an odd integer and consider the circulant graph $C_n^{1,2}$.  Letting $\sigma_1$ and $\sigma_2$ be involutions on $C_n^{1,2}$ as defined in Example \ref{Ex:circ}, we can see that $C_n^{1,2}/\sigma_1$ and $C_n^{1,2}/\sigma_2$ are both isomorphic to the graph $H_k$ as defined in Figure \ref{F:circmod}.

\begin{figure}
\begin{center}
\begin{tikzpicture}
[scale=.5,auto=left,every node/.style={circle,fill=blue!20}]

\node (v0) at (0,0) {$v_1$};
\node (v1) at (3,0) {$v_2$};
\node (v2) at (6,0) {$v_3$};
\node (v3) at (9,0) {$v_4$};
\node (vk1) at (12,0) {$v_k$};
\node (vk) at (15,0) {$v_\infty$};

\foreach \from/\to in {v0/v1,v1/v2,v2/v3,vk1/vk}
    \draw (\from) -- (\to);

\draw[dotted] (v3) -- (vk1);

\foreach \from/\to in {v0/v2,vk1/vk,v3/v1}
    \draw (\from) to[bend left=45] (\to);

\foreach \from/\to in {v2/vk1}
    \draw[dotted] (\from) to[bend left=45] (\to);
\end{tikzpicture}

\caption{The graph $H_k = C_n^{1,2}/\langle\sigma\rangle$}
\label{F:circmod}
\end{center}
\end{figure}
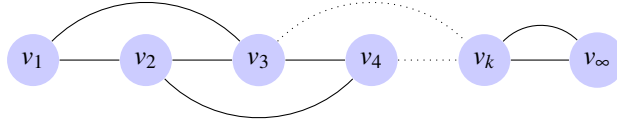

Moreover, $C_n^{1,2}/\langle\sigma_1\sigma_2\rangle$ and $C_n^{1,2}/D_n$ are both a single point and therefore have trivial Jacobians.  In particular, it follows from Theorems \ref{T:kernel} and \ref{T:quotient} that the following sequence is exact:
\[1 \rightarrow \Jac(H_k) \oplus \Jac(H_k) \rightarrow \Jac(C_n^{1,2}) \rightarrow \ZZ/n\ZZ \rightarrow 1\]

In order to compute the Jacobian of the graph $H_k$, we first note that it is cyclic.  In particular, for any divisor $\delta$, we note that after performing $\delta(v_1)$ `borrowing' operations at vertex $v_3$, the divisor $\delta$ will be equivalent to a divisor $\delta_1$ so that $\delta_1(v_1)=0$.  One can repeat this procedure for all $j < k$ to obtain equivalent divisors $\delta_j$ so that $\delta_j(v_i)=0$ for all $i \le j$.  In particular, $\delta_{k-1}$ will only have nonzero values at the vertices $v_k$ and $v_\infty$ and therefore must be a multiple of the divisor $\gamma$ which has $\gamma(v_k)=-\gamma(v_\infty)=1$ and $\gamma(v_i)=0$ for $i<k$.  This shows that $\gamma$ is a generator of $\Jac(H_k)$.

We determine the order of $\Jac(H_k)$ by counting the number of spanning trees of the graph $H_k$; we denote this number by $a_k$.  In particular, one can easily compute that $a_1=2$  and $a_2=5$.  Moreover, any spanning tree of $H_k$ will fall into one of the following three categories
\begin{itemize}
\item A spanning tree on the vertices $\{v_2,\ldots,v_k,v_\infty\}$ and the edge between $v_1$ and $v_2$.  There are $a_{k-1}$ such trees.
\item A spanning tree on the vertices $\{v_2,\ldots,v_k,v_\infty\}$ and the edge between $v_1$ and $v_3$. There are $a_{k-1}$ such trees.
\item A spanning forest on the vertices $\{v_2,\ldots,v_k,v_\infty\}$ such that $v_2$ and $v_3$ are in different components in addition to the two edges from $v_1$ to $v_2$ and $v_3$.  There are $a_{k-1}-a_{k-2}$ such trees.
\end{itemize}
In particular, we have that $a_k=3a_{k-1}-a_{k-2}$.  This shows that $a_k=F_{2k+1}=F_n$ is the $n^{th}$ Fibonacci number, using the convention that $F_1=F_2=1$.  It follows that the number of spanning trees on $C_n^{1,2}$ is equal to $nF_n^2$, agreeing with a well-known result originally proven in \cite{KG}.  Moreover, if $(n,F_n)=1$ then it follows from our result that $\Jac(C_n^{1,2}) \cong (\ZZ/F_n\ZZ) \oplus (\ZZ/F_n\ZZ) \oplus (\ZZ/n\ZZ)$.

\vskip .1in

\subsection{Klein-Four Example}

Let $\GG$ be the graph in Figure \ref{F:Klein}, which admits a harmonic $(\ZZ/2\ZZ)^2$-action defined by the involutions $\sigma_1$ which swaps the vertices $x_1$ and $x_2$ and $\sigma_2$ which swaps $a_i$ and $b_i$.

\begin{figure}[h]
\begin{center}
\begin{tikzpicture}
  [scale=.5,auto=left,every node/.style={circle,fill=blue!20}]
  \node (l) at (0,0) {$x_1$};
  \node (r) at (6,0)  {$x_2$};
  \node (u1) at (3,1.5) {$a_1$};
  \node (u2) at (3,4) {$a_2$};
  \node (d1) at (3,-1.5) {$b_1$};
  \node (d2) at (3,-4) {$b_2$};

  \foreach \from/\to in {l/u1,l/u2,l/d1,l/d2,r/u1,r/u2,r/d1,r/d2}
    \draw (\from) -- (\to);
\end{tikzpicture}

\caption{A graph admitting a $(\ZZ/2\ZZ)^2$-action}
\label{F:Klein}
\end{center}
\end{figure}
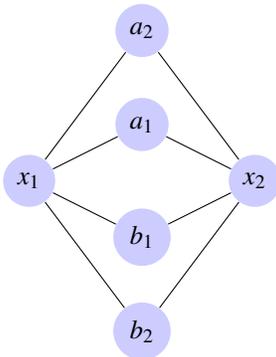

In this case, one can easily compute that $\GG/\sigma_1$ is a tree, $\GG/\sigma_2$ is isomorphic to a square, and $\GG/\sigma_1\sigma_2$ is isomorphic to two $2$-cycles sharing a single vertex.  In particular, setting $J = \Jac(\GG/\sigma_1)+\Jac(\GG/\sigma_2)+\Jac(\GG/\sigma_1\sigma_2)$, Theorem \ref{T:kernel} implies that we have an exact sequence
\[ 1 \rightarrow (\ZZ/2\ZZ)^2 \rightarrow (\ZZ/2\ZZ)^2 \oplus \ZZ/4\ZZ \rightarrow J \rightarrow 1 \]

\noindent and Theorem \ref{T:quotient} implies that we have an exact sequence
\[ 1 \rightarrow J \rightarrow \Jac(\GG) \rightarrow (\ZZ/2\ZZ)^2 \oplus (\ZZ/2\ZZ) \rightarrow 1 \]

One can in fact show in this case that the exact sequences split in such a way that gives $\Jac(\GG) \cong (\ZZ/2\ZZ)^2 \oplus (\ZZ/8\ZZ)$.

\subsection{Concentric Polygons}

For this example, we consider the graph $\GG_4$ defined in Figure \ref{F:octagon}. We let $\sigma_1$ be the involution defined by reflection in the line through the vertices $z_1$ and $z_3$ and we let $\sigma_2$ be defined by reflection across the vertical axis.

\begin{figure}[h]
\begin{center}

\begin{tikzpicture}
  [scale=.7,auto=left,every node/.style={circle,fill=blue!20}]
  \node (c1) at (1,1) {$z_1$};
  \node (c2) at (3,1)  {$z_2$};
  \node (c3) at (3,-1) {$z_3$};
  \node (c4) at (1,-1) {$z_4$};
  \node (o1) at (1,3) {$x_1$};
  \node (o2) at (3,3) {$y_2$};
  \node (o3) at (5,1){$x_2$};
  \node (o4) at (5,-1){$y_3$};
  \node (o5) at (3,-3){$x_3$};
  \node (o6) at (1,-3){$y_4$};
  \node (o7) at (-1,-1){$x_4$};
  \node (o8) at (-1,1){$y_1$};

  \foreach \from/\to in {c1/c2,c2/c3,c3/c4,c4/c1,o1/o2,o2/o3,o3/o4,o4/o5,o5/o6,o6/o7,o7/o8,o8/o1,c1/o1,c1/o8,c2/o2,c2/o3,c3/o4,c3/o5,c4/o6,c4/o7}
    \draw (\from) -- (\to);
\end{tikzpicture}
\caption{The graph $\GG_4$}
\label{F:octagon}
\end{center}
\end{figure}

One can easily check that $\GG_4/D_n$ is a tree, and therefore it follows from Theorems \ref{T:kernel} and \ref{T:quotient} that $\Jac(\GG_4/\sigma_1) \oplus \Jac(\GG_4/\sigma_2) \oplus \Jac(\GG_4/\sigma_1\sigma_2)$ is a subgroup of $\Jac(\GG_4)$ whose quotient is isomorphic to $\ZZ/4\ZZ$.  In figure \ref{F:quotients} we display the quotient graphs $\GG_4/\sigma_1$, $\GG_4/\sigma_2$, and $\GG_4/\sigma_1\sigma_2$.

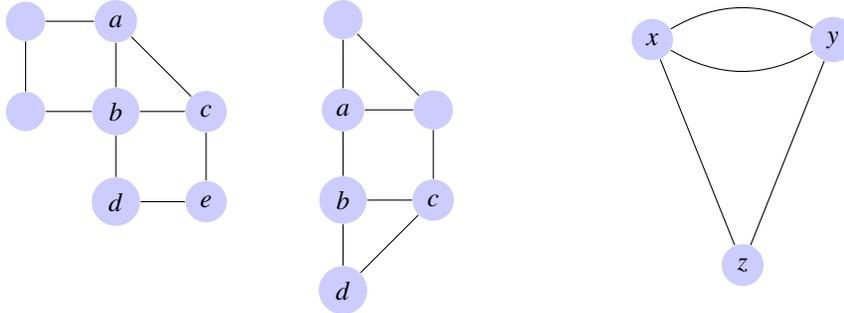
\begin{figure}[h]
\begin{multicols}{3}

\begin{tikzpicture}
  [scale=.6,auto=left,every node/.style={circle,fill=blue!20}]
  \node (c1) at (1,1) {\,\,\,\,};
  \node (c2) at (3,1)  {$b$};
  \node (c3) at (3,-1) {$d$};
  \node (o1) at (1,3) {\,\,\,\,};
  \node (o2) at (3,3) {$a$};
  \node (o3) at (5,1){$c$};
  \node (o4) at (5,-1){$e$};

  \foreach \from/\to in {c1/c2,c2/c3,o1/o2,o2/o3,o3/o4,c1/o1,c2/o2,c2/o3,c3/o4}
    \draw (\from) -- (\to);
\end{tikzpicture}

\begin{tikzpicture}
  [scale=.6,auto=left,every node/.style={circle,fill=blue!20}]
  \node (c2) at (3,1)  {$a$};
  \node (c3) at (3,-1) {$b$};
  \node (o2) at (3,3) {\,\,\,\,};
  \node (o3) at (5,1){\,\,\,\,};
  \node (o4) at (5,-1){$c$};
  \node (o5) at (3,-3){$d$};

  \foreach \from/\to in {c2/c3,o2/o3,o3/o4,o4/o5,c2/o2,c2/o3,c3/o4,c3/o5}
    \draw (\from) -- (\to);
\end{tikzpicture}

\begin{tikzpicture}
  [scale=.6,auto=left,every node/.style={circle,fill=blue!20}]
  \node (x) at (1,1) {$x$};
  \node (y) at (5,1)  {$y$};
  \node (z) at (3,-4) {$z$};
  \foreach \from/\to in {x/z,y/z}
    \draw (\from) -- (\to);
   \draw (x) to[bend left=30] (y);
   \draw (y) to[bend left=30] (x);

\end{tikzpicture}

\end{multicols}
\caption{Quotients of the graph $\GG_4$}
\label{F:quotients}

\end{figure}

By successively doing the appropriate number of `borrowing' operations from the vertices $a,b,c,d,e$, any divisor $\delta$ on $\GG_4/\sigma_1$ is equivalent to a divisor that has zeroes on all vertices other than $d$ and $e$, showing that $\Jac(\GG_4/\sigma_1)$ is cyclic.  Moreover, one can directly count that the number of spanning trees on $\GG_4/\sigma_1$ is $40$ so $\Jac(\GG_4/\sigma_1) \cong \ZZ/40\ZZ$. A similar argument can be used to show that $\Jac(\GG_4/\sigma_2)$ is the cyclic group of order $30$.  Finally, one can compute directly that $\Jac(\GG_4/\sigma_1\sigma_2) \cong \ZZ/5\ZZ$.  Therefore, we have that $(\ZZ/40\ZZ) \oplus (\ZZ/30\ZZ) \oplus (\ZZ/5\ZZ) \subseteq \Jac(\GG_4)$ with quotient equal to $\ZZ/4\ZZ$.  We note that in this example one can compute that $\Jac(\GG_4) \cong (\ZZ/160\ZZ) \oplus (\ZZ/30\ZZ) \oplus (\ZZ/5\ZZ)$ showing that in general the quotient does not split as a direct sum.

\bibliographystyle{amsplain}
\bibliography{critgps}

\end{document}